\documentclass{amsart}

\usepackage{graphicx}
\usepackage{amsmath,amssymb,amsthm}
\usepackage{lineno}

\newtheorem{theorem}{Theorem}[section]
\newtheorem{proposition}[theorem]{Proposition}
\newtheorem{lemma}[theorem]{Lemma}
\newtheorem{corollary}[theorem]{Corollary}

\theoremstyle{definition}
\newtheorem{definition}{Definition}

\theoremstyle{remark}
\newtheorem{remark}{Remark}

\begin{document}

\title[Two-tone colorings and dihedral representations]{Two-tone colorings and surjective dihedral representations for links}

\author[K.~Ichihara]{Kazuhiro Ichihara}
\address{Department of Mathematics, College of Humanities and Sciences, Nihon University, 3-25-40 Sakurajosui, Setagaya-ku, Tokyo 156-8550, Japan}
\email{ichihara.kazuhiro@nihon-u.ac.jp}

\author[K.~Ishikawa]{Katsumi Ishikawa}
\address{Research Institute for Mathematical Sciences, Kyoto University, Kyoto 606-8502, Japan}
\email{katsumi@kurims.kyoto-u.ac.jp}

\author[E.~Matsudo]{Eri Matsudo}
\address{The Institute of Natural Sciences, Nihon University, 3-25-40 Sakurajosui, Setagaya-ku, Tokyo 156-8550, Japan}
\email{matsudo.eri@nihon-u.ac.jp}

\author[M.~Suzuki]{Masaaki Suzuki}
\address{Department of Frontier Media Science, Meiji University, 4-21-1 Nakano, Nakano-ku, Tokyo, 164-8525, Japan}
\email{mackysuzuki@meiji.ac.jp}

\thanks{This work was supported by JSPS KAKENHI Grant Numbers 22K03301,  20K14309, 20K03596.}

\subjclass[2020]{57K10}
\keywords{link, coloring, dihedral group}
\date{\today}

\begin{abstract}
It is well-known that a knot is Fox $n$-colorable for a prime $n$ if and only if the knot group admits a surjective homomorphism to the dihedral group of degree $n$. 
However, this is not the case for links with two or more components. 
In this paper, we introduce a two-tone coloring on a link diagram, and give a condition for links so that the link groups admit surjective representations to the dihedral groups. 
In particular, it is shown that the link group of any link with at least 3 components admits a surjective homomorphism to the dihedral group of arbitrary degree. 
\end{abstract}

\maketitle

\section{Introduction}\label{sec1}

One of the most well-known invariants of knots in 3-space must be the Fox's 3-colorablity. (See Remark~\ref{rmk11} for the definition of the Fox $n$-coloring.) 
In general, it is known that a knot is Fox $n$-colorable 
for a prime $n \ge 3$ if and only if the knot group admits a surjective homomorphism to the dihedral group $D_n$ of degree $n$. 
For instance, it is stated in \cite[Chap. VI, Exercises, 6, pp.92--93]{CrowellFox1963}. 
However, this is not the case for links with two or more components. 
In fact, some examples are given in \cite{IchiharaMatsudo2022} for $D_3$-coloring, which is the coloring by the symmetric group of degree three. 
For example, by the results in \cite[Theorem 1.2]{IchiharaMatsudo2022}, the link group of the torus link $T(2,q)$ admits a surjective homomorphism to $D_3$ if $q\equiv 0 \pmod {4}$. 
On the other hand, $T(2,q)$ is Fox 3-colorable if and only if $q\equiv 0\pmod 3$.

We remark that, although there are numerous papers studying the Fox colorings (cf. \cite{Przytycki1998, CarterSilverWilliams2014}), it seems that the relationship between the Fox colorings on links with two or more components and the surjective homomorphisms of the link groups to the dihedral groups has not been discussed, as far as the authors know. 

In this paper, we introduce a two-tone coloring on a link diagram, and give a condition for links  which guarantees that the link groups admit surjective homomorphisms to the dihedral groups. 
In particular, we show that the link group of any link with at least 3 components admits a surjective homomorphism to the dihedral group of arbitrary degree. 

\begin{remark}\label{rmk11}
Recall that a \textit{Fox $n$-coloring} on a link diagram $D$ is defined as a map $\Gamma:\{$arcs of $D\}\rightarrow \{ 0, 1, \dots , n-1 \}$, 
satisfying $2\Gamma(x)\equiv \Gamma(y)+\Gamma(z) \pmod n$ at each crossing of $D$ with the over arc $x$ and the under arcs $y$ and $z$. 
It is well-known that, for $n \ge 3$, a link is Fox $n$-colorable, i.e., a diagram of the link admits a non-trivial Fox $n$-coloring (a coloring with at least two colors), if and only if $\det (L) =0$ or $(n, \det(L)) \ne 1$, where $\det(L)$ denotes the determinant of the link. 
See \cite[Proposition 2.1]{LopezMatias2012} for example. 
Also a condition for knot groups to admit a surjective homomorphism to the dihedral groups in terms of the homology of the double branched covering is known. 
See \cite[14.8]{BurdeZieschang} for example. 
\end{remark}

To state our results, we prepare some notations. 
Let $D_n$ be the dihedral group of degree $n$. 
It is well-known that $D_n$ has the following presentation with $e$ the identity element: 
\[
D_n 
=\left< a, b \mid a^2 = b^n = (ab)^2 = e \right> .
\]
Note that any element in $D_n$ is represented as $a^x b^y$ ($x=0,1$ and $0 \le y \le n-1$). 
Thus, by setting $a_i := a b^i$ ($0 \le i \le n-1$) and $b_j := b^j$ ($1 \le j \le n-1$), 
we see that $D_n = \{ e, a_0, a_1, \dots , a_{n-1}, b_1, \dots, b_{n-1} \}$ as a set. 
In geometric viewpoint, the $a_i$'s represent reflections and $b_j$'s represent rotations as the symmetries of a regular polygon ($n$-gon). 

In the following, let $L$ be an oriented link in the 3-sphere $S^3$ with a link diagram $D$. 
We call a map $\Gamma:\{ \mbox{arcs on $D$} \} \rightarrow D_n$ a {\it $D_n$-coloring} on $D$ if it satisfies $\Gamma(x) \Gamma(z) = \Gamma(y) \Gamma(x)$ 
(respectively, $\Gamma(z) \Gamma(x) = \Gamma(x) \Gamma(y)$) in $D_n$ at each positive (resp. negative) crossing on $D$, where $x$ denotes the over arc, $y$ and $z$ the under arcs at the crossing supposing $y$ is the under arc before passing through the crossing and $z$ is the other. (See Figure~\ref{Fig/condition}.)

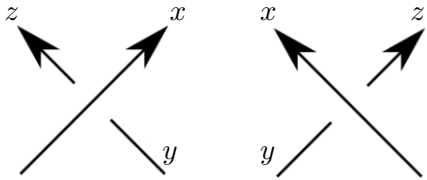
\begin{figure}[htbt]
\centering
  {\unitlength=1mm
  \begin{picture}(100,23)
   \put(25,0){\includegraphics[width=.44\textwidth]{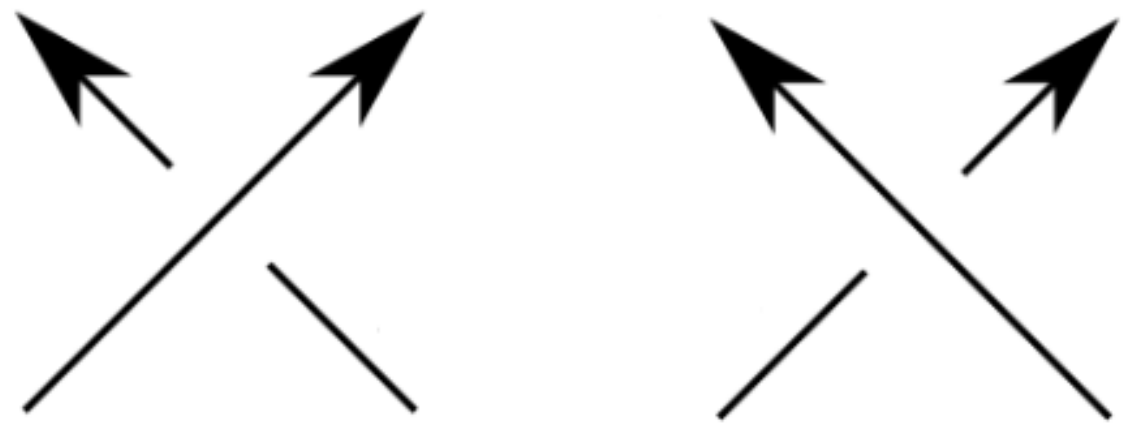}}
   \put(46,21.5){{\large $x$}}
   \put(24,21.5){{\large $z$}}
   \put(45,3){{\large $y$}}
   \put(58,21.5){{\large $x$}}
   \put(78,21.5){{\large $z$}}
   \put(58,3){{\large $y$}}
  \end{picture}}
\caption{Positive and negative crossings}\label{Fig/condition}
\end{figure}

\begin{remark}\label{rmk2}
The $D_n$-colorings and the Fox $n$-colorings are related in terms of representations of link groups to $D_n$ as follows. 
For a link diagram $D$ with $c$ crossings of a link $L$, 
set $g_1, \dots, g_c$ the Wirtinger generators of the link group $G_L$, i.e., $G_L = \pi_1 (S^3 - L)$. 
Then a $D_n$-coloring on $D$ corresponds to a map $\{ g_1 , \dots , g_c \} \to D_n$ which extends to a homomorphism of $G_L$ to $D_n$. 
When a $D_n$-coloring sends $g_k$'s to $a_i$'s (reflections, $0 \le i \le n-1$) in $D_n$, it induces a map $\{ \mathrm{arcs~of~} D \} \to \{ 0, 1, \dots, n-1 \}$, which gives a Fox $n$-coloring. 
Note that even if a link admits a nontrivial Fox $n$-coloring, it may not induce a surjective homomorphism from $G_L$ to $D_n$. 
See the example illustrated in Figure~\ref{Fig4}. 
In this case, the image of the Wirtinger generators by the homomorphism induced by the Fox $4$-coloring is the set $\{ a_0, a_2 \} \subset D_4$, but the elements $a_0$ and $a_2$ do not generate $D_4$. 
Thus the induced homomorphism is not surjective. 
\begin{figure}[htbt]
\centering
  {\unitlength=1mm
  \begin{picture}(90,40)
   \put(25,0){
  \includegraphics[width=.275\textwidth]{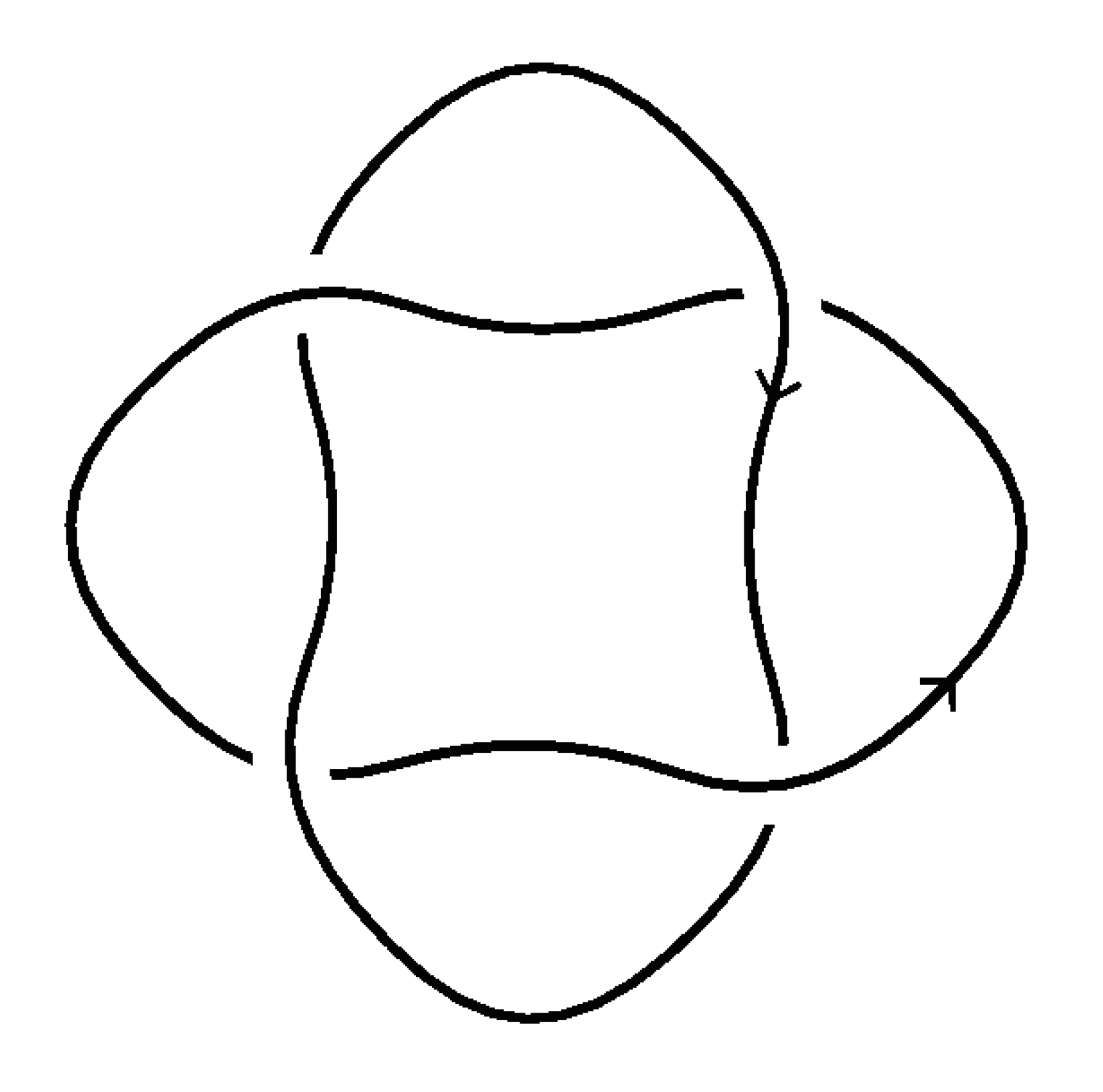} }
   \put(29,24){$0$}
   \put(56,24){$0$}
   \put(49,2){$2$}
   \put(49,31){$2$}
  \end{picture}}
\caption{Fox 4-colorable link}\label{Fig4}
\end{figure}
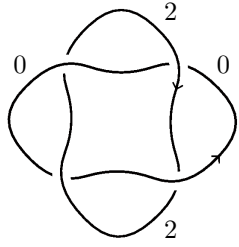
\end{remark}

The following is our key definition. 

\begin{definition}
Let $\Gamma$ be a $D_n$-coloring on a link diagram $D$ of an oriented link $L$. 
We say that $\Gamma$ is \textit{two-tone} if $\mathrm{Im} (\Gamma)$ does not contain the trivial element, i.e. $ e \not\in \mathrm{Im} (\Gamma)$, and $\mathrm{Im} (\Gamma) \cap \{a_0,\dots,a_{n-1} \} \ne \emptyset $ and $\mathrm{Im} (\Gamma) \cap \{ b_1,\dots, b_{n-1} \} \ne \emptyset $, that is,  the coloring uses colors from both $\{ a_i \} $ and $\{ b_j \}$. 
We say that a link is \textit{two-tone $D_n$-colorable} if, with some orientation, it has a diagram $D$ admitting a two-tone $D_n$-coloring. 
\end{definition}

Note that two-tone $D_n$-colorability is independent of the choice of orientations for links. 

An example of a two-tone $D_n$-colorable link is the pretzel link $P(6,6,6)$ which admits a two-tone $D_m$-coloring if $m \ge 4$. 
See Figure~\ref{Fig0} for the case where $m \ge 8$. 
\begin{figure}[htbt]
\centering
  {\unitlength=1mm
  \begin{picture}(90,60)
   \put(25,0){\includegraphics[width=.375\textwidth]{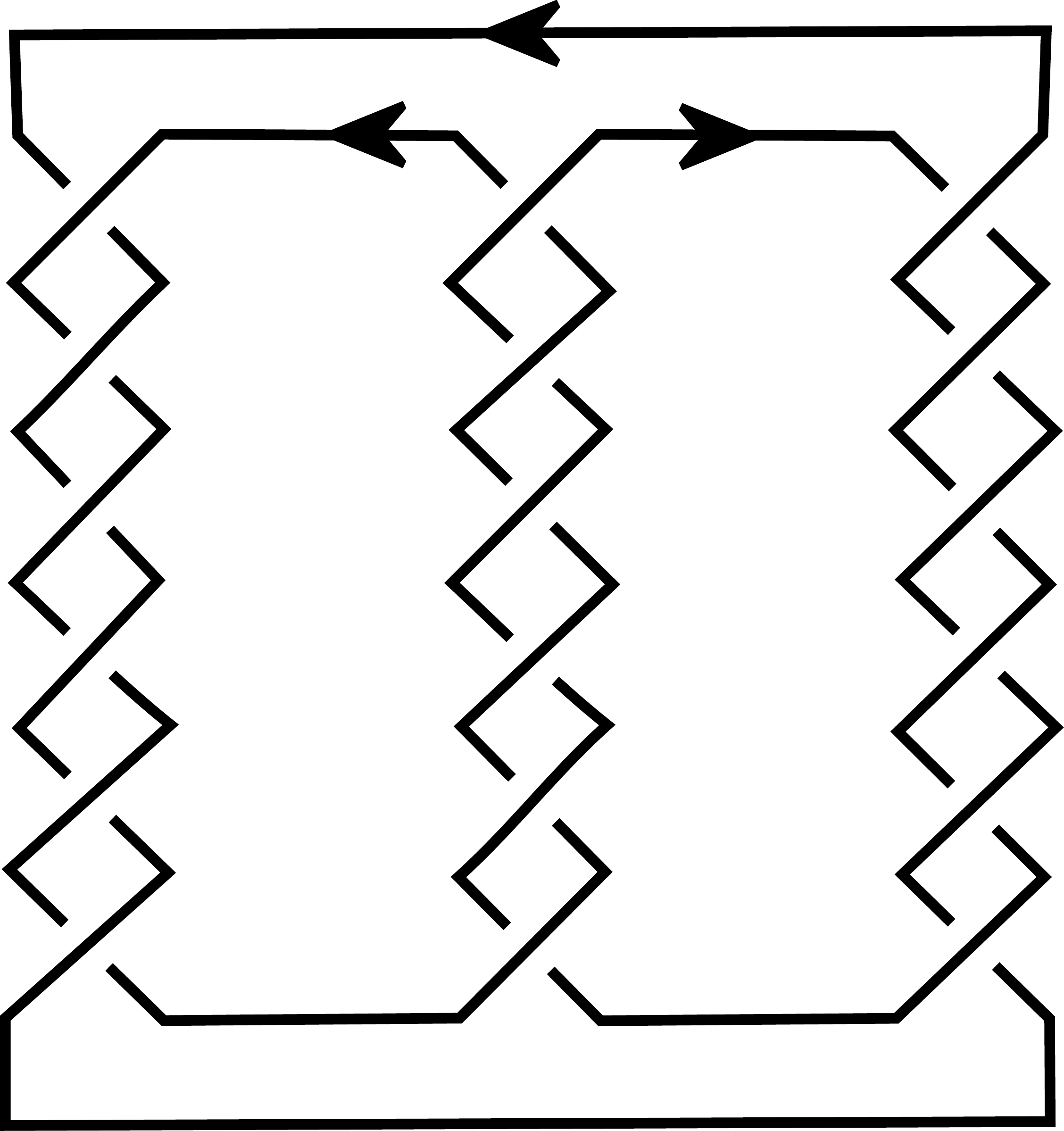}}
   \put(52,50){$b_{m-3}$}
   \put(33,37){$b_3$}
   \put(33,24){$b_{m-3}$}
   \put(72.5,18){$b_{m-3}$}
   \put(72.5,30.5){$b_3$}
   \put(40,-3){$b_3$}
   \put(35,45){$a_0$}
   \put(60,45){$a_1$}
   \put(33,30.5){$a_6$}
   \put(33,18){$a_0$}
   \put(37,6){$a_6$}
   \put(53,18){$a_5$}
   \put(53,24){$a_4$}
   \put(53,30.5){$a_3$}
   \put(53,37){$a_2$}
   \put(56,6){$a_7$}
   \put(72.5,37){$a_7$}
   \put(72.5,24){$a_1$}
  \end{picture}}
\caption{A two-tone $D_m$-coloring on $P(6,6,6)$ for $m \ge 8$}\label{Fig0}
\end{figure}
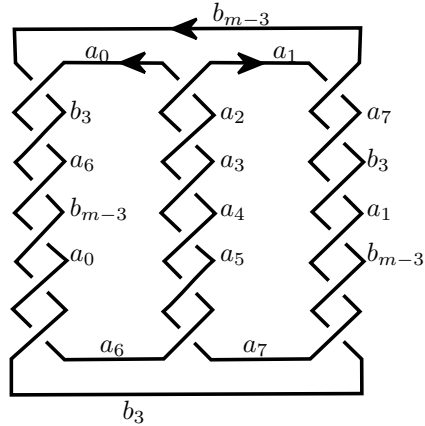

\medskip

Now the following are our main results. 
Here $D_\infty$ denotes the group presented by $\left< a, b \mid a^2 = (ab)^2 = e \right>$, and a two-tone $D_\infty$-coloring for a link is defined in the same way as above. 

\begin{theorem}\label{Thm1}
For a 2-component link $L = \ell_1 \cup \ell_2$, the following are equivalent.
\begin{itemize}
\item[(i)] $lk( \ell_1, \ell_2)$ is even. 
\item[(ii)] $L$ is two-tone $D_n$-colorable for some odd $n \ge 3$. 
\item[(iii)] $L$ is two-tone $D_\infty$-colorable.
\item[(iv)] The link group $G_L$ admits a surjective homomorphism to $D_n$ for every $n \ge 3$.
\item[(v)] The link group $G_L$ admits a surjective homomorphism to $D_\infty$.
\end{itemize}
\end{theorem}

\begin{remark}\label{rem3}
By considering the natural embedding of $D_n$ into $D_{2n}$, we see that the condition (ii) in Theorem~\ref{Thm1} is equivalent to 
that $L$ admits a two-tone $D_n$-coloring for some even $n \ge 3$ that assigns $b_i$ with $i \ne n/2$ to some arcs. 
We also remark that (iii) in Theorem~\ref{Thm1} does not imply that $L$ is two-tone $D_n$-colorable for every odd $n \ge 3$. 
Actually even if there is a two-tone $D_\infty$-coloring on a diagram of a link $L$, the coloring may not  give a two-tone $D_n$-coloring for some $n$, but a Fox $n$-coloring on a sub-diagram of $L$. 
For example, pretzel links of type $(m,2,m,2)$ with odd $m$ admit a two-tone $D_\infty$-coloring on a diagram, but no two-tone $D_m$-colorings. 
\end{remark}

On the other hand, for 2-component links with odd linking numbers, we have the following. 

\begin{theorem}\label{Thm2}
Let $L = \ell_1 \cup \ell_2$ be a 2-component link with $lk( \ell_1,\ell_2)$ odd.
Then the following hold. 
\begin{itemize}
\item[(i)] The link $L$ admits no two-tone $D_n$-colorings for any odd $n \ge 3$. 
\item[(ii)] 
If the link group $G_L$ admits a surjective homomorphism to $D_n$ for $n \ge 3$, then the homomorphism is induced from a Fox $n$-coloring on $\ell_1$, $\ell_2$ or $L$, i.e., the homomorphism sends a meridional element in $G_L$ to the trivial element or a reflection in $D_n$. 
\end{itemize}
\end{theorem}

For the links with at least 3 components, interestingly, the following holds. 

\begin{theorem}\label{Thm3}
Let $L$ be a link with at least 3 components. 
Then the link group $G_L$ admits a surjective homomorphism to $D_n$ for every $n \ge 3$. 
\end{theorem}

We remark that even if  the link group $G_L$ admits a surjective homomorphism to $D_n$  for every $n \ge 3$, the link $L$ may not be two-tone $D_n$-colorable for every $n \ge 3$. 
For example, pretzel links of type $(2m,2m,2m)$ with odd $m$ admit no two-tone $D_m$-colorings.  

As a corollary of the theorems, we have the following. 

\begin{corollary}\label{CorG}
If a link $L$ is two-tone $D_m$-colorable for some odd  $m$, then $G_L$ admits a surjective homomorphism to $D_n$ for every $n \ge 3$. 
If $G_L$ admits a surjective homomorphism to $D_n$ for some $n$, then $L$ 
contains a two-tone $D_n$-colorable sub-link or a Fox $n$-colorable sub-link. 
\end{corollary}

On the other hand, even if a link $L$ is known to be two-tone $D_n$-colorable for some $n$, finding a two-tone $D_n$-coloring on a given diagram of $L$, or, finding a surjective homomorphism of $G_L$ to $D_n$, is a tedious task in general. 
The next proposition and its proof give a simple way to find a two-tone $D_n$-coloring on some link diagrams for any $n \ge 3$. 

\begin{proposition}\label{prop1}
Suppose that there exists a trivial component $\ell_0$ of a link $L$ and that $lk (\ell_0 , \ell)$ is even for every component $\ell \subset L - \ell_0$. 
Then any diagram of $L$ admits a two-tone $D_n$-coloring  for every odd $n \geq 3$ which assigns the arcs on $\ell_0$ to $a_i$'s and the other arcs to $b_j$'s.
\end{proposition}

\section{Properties of $D_n$-coloring}
In this section, we study some properties of $D_n$-colorings, and give lemmas which will be used in the remaining sections. 
In the following, we set $A_n := \{ a_i \}$ and $B_n:=\{ b_j \}$ for $D_n$. 

\begin{lemma}\label{lem21}
Let $\Gamma$ be a $D_n$-coloring on a diagram $D$ of an oriented link $L$ in $S^3$. 
At a crossing on $D$, $x$ denotes the over arc, and $y$ and $z$ the under arcs at the crossing supposing that $y$ (resp. $z$) is the under arc before (resp. after) passing through the crossing. 
Then the following hold. 
\begin{enumerate}
\item
$\Gamma(y)$ and  $\Gamma(z)$ are both in $A_n$ or both in $B_n$. 
\item
If $\Gamma(x) \in B_n$ and $\Gamma(y) \in B_n$, then $\Gamma(z) = \Gamma (y)$. 
\item
If $\Gamma(x) = a_i$ and $\Gamma(y) = a_{i'}$, then $\Gamma(z) = a_k$ and $ k \equiv 2 i' - i \pmod n$.
\item
If $\Gamma(x) = a_i$ and $\Gamma(y) = b_j$, then $\Gamma(z) = b_k$ and $ k \equiv n-j  \pmod n$. 
\item
If $\Gamma(x) = b_j$ and $\Gamma(y) = a_j$, then $\Gamma(z) = a_k$ and $ k \equiv i + 2j $ (resp. $ k \equiv i - 2j$ ) $ \pmod n$ if the crossing is a positive (resp. negative) crossing. 
\end{enumerate}
\end{lemma}

\begin{proof}
By definition of a $D_n$-coloring, $\Gamma(y)$ and  $\Gamma(z)$ are conjugate in $D_n$, and from this, (1) holds. 
We give a proof of the case (4) when the crossing is a positive crossing. 
The others (2), (3), (5) are proved in the same way. 
Suppose that $\Gamma(x) = a_i$ and $\Gamma(y) = b_j$. 
By definition of a $D_n$-coloring, we have the following. 
\begin{align*}
\Gamma(z) &= (a_i)^{-1} b_j a_i = b^{n-i} a^{-1} b^j a b^i \\
&= a b^{i+j-n} a b^i = b^{n-i-j+i} = b^{n-j}=b_{n-j}
\end{align*}
Thus $ \Gamma(z) = b_k$ and $ k \equiv n-j  \pmod n$ holds. 
\end{proof}

\begin{remark}\label{rmk21}
Note that (1) in the lemma implies that all the strands on a diagram of a particular component must be colored by $a_i$'s or by $b_j$'s for a given $D_n$-coloring. 
Also note that the tone of the colors for a particular component is independent of the choice of a diagram: 
If all the strands on a diagram for a particular component are colored by $b_j$'s by a $D_n$-coloring, then all the strands for the component are also colored by $b_j$'s on any diagram by the $D_n$-coloring obtained by performing Reidemeister moves. 
We will use these facts in the rest of the paper repeatedly. 
\end{remark}

\begin{lemma}\label{lem22}
Let $\Gamma$ be a $D_n$-coloring on a diagram $D$ of an oriented link $L$ in $S^3$. 
Let $x,y,z,w$ be either the arcs depicted in Figure~\ref{Fig22} (left), or the arcs depicted in Figure~\ref{Fig22} (right). 
If $\Gamma(x) = a_i$ and $\Gamma(y) = b_j$, then $\Gamma(z) = a_k$ with $ k \equiv i-2j  \pmod n$ and $\Gamma(w) = b_l$ with $ l \equiv n-j  \pmod n$.
\end{lemma}

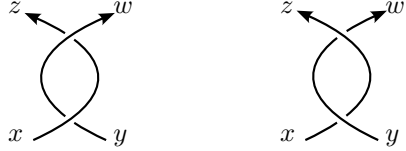
\begin{figure}[htbt]
\centering
  {\unitlength=1mm
  \begin{picture}(35,17)
   \put(16,0){
  \includegraphics[width=.095\textwidth]{ 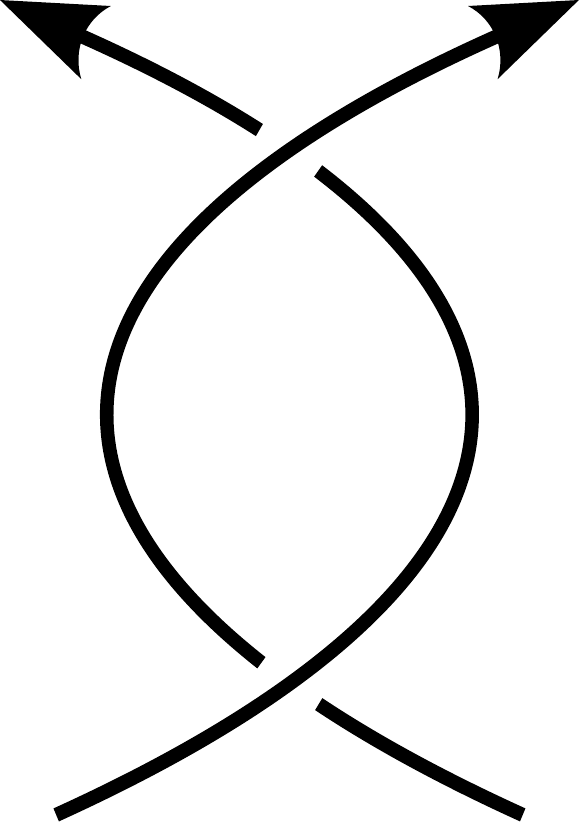} }
   \put(15,0){$x$}
   \put(29,0){$y$}
   \put(15,17){$z$}
   \put(29,17){$w$}
  \end{picture}}
  {\unitlength=1mm
  \begin{picture}(35,17)
   \put(16,0){
  \includegraphics[width=.095\textwidth]{ 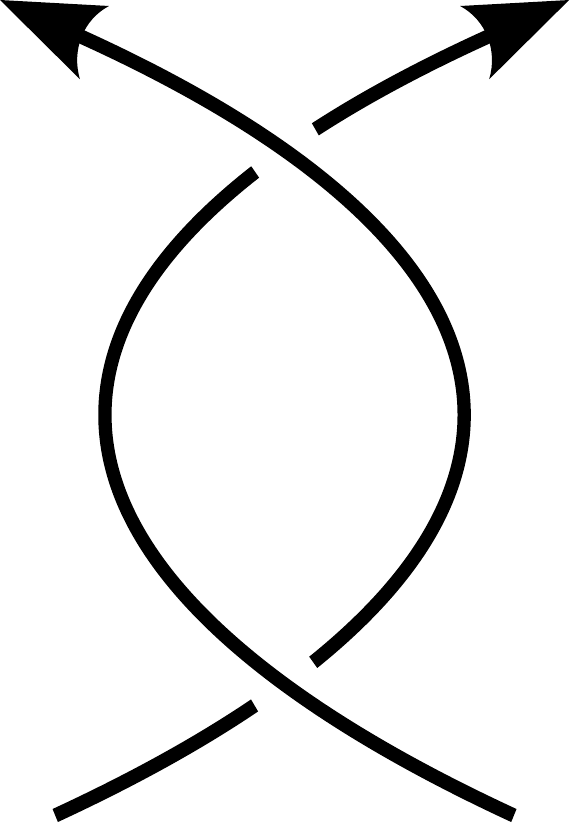} }
   \put(15,0){$x$}
   \put(29,0){$y$}
   \put(15,17){$z$}
   \put(29,17){$w$}
  \end{picture}}
\caption{\ A positive full twist (left). \quad A negative full twist (right)}\label{Fig22}
\end{figure}

\begin{proof}
We only give a proof for the positive full twist case. 
A proof for the other case is similar. 
In that case, by Lemma~\ref{lem21}(4), $\Gamma(w) = b_l$ with $ l \equiv n-j  \pmod n$ since $\Gamma(x) = a_i$ and $\Gamma(y) = b_j$. 
Then, by Lemma~\ref{lem21}(5), $\Gamma(z) = a_k$ and $ k \equiv i + 2(n-j) \equiv i -2j  \pmod n$ since $\Gamma(w) = b_l$ with $ l \equiv n-j  \pmod n$ and $\Gamma(x) = a_i$. 
\end{proof}

\section{Two-tone colorings and surjective homomorphisms to $D_\infty$}
In this section, 
we give a key proposition to prove the theorems. 

In the following, let $lk(L,L')$ denote the (total) linking number of oriented links $L, L'$, i.e., $lk(L,L') = \sum_{\ell \subset L, \ell' \subset L'} lk(\ell, \ell')$.  
The linking number is calculated for the link with arbitrarily chosen orientations. 
Note that the parity of such a linking number is independent of the choice of orientations.

\begin{proposition}\label{prop}
Suppose that a link L contains a component $\ell_0$ with $lk(\ell_0, L')$ even and $\det(L') \neq 0$, where $L' = L - \ell_0$.
Then 
$L$ admits a two-tone $D_\infty$-coloring that induces a surjective homomorphism from $G_L$ to $D_\infty$.
\end{proposition}

\begin{proof}
\par Let $p: X \to S^3 - L'$ be the 
double covering on the total linking number with $L'$, and $\bar{p}: M \to S^3$ the 
double branched covering. 
Let $\tilde{K} = K_1 \cup K_2$ denote the inverse image $p^{-1}(\ell_0) \subset X$; because $lk(\ell_0, L')$ is even, $\tilde{K}$ is a $2$-component link in $X$ (or in $M$). 
We shall construct a surjective group homomorphism $\pi_1(M - \tilde{K}) \to \mathbb{Z}$ and extend the composition $\pi_1(X - \tilde{K}) \to \pi_1(M - \tilde{K}) \to \mathbb{Z}$ to obtain a $D_\infty$-coloring.
\par Taking a regular neighborhood $N$ of $\tilde{K}$, we consider the Mayer-Vietoris exact sequence for $M = N \cup (M - \tilde{K})$:
$$H_2(M) \to H_1(N - \tilde{K}) \to H_1(N) \oplus H_1(M - \tilde{K}) \to H_1(M) \to H_0(N - \tilde{K})$$
is exact. The rightmost map is zero as usual and the leftmost one is also zero because $\det (L') \neq 0$ (hence $|H_1(M)| = |\det (L')| < \infty$); by the Poincar\'{e} duality $H_2(M) \cong H^1(M; \mathbb{Z}) = 0.$ Thus, we obtain a short exact sequence
$$0 \to H_1(N - \tilde{K}) \to H_1(N) \oplus H_1(M - \tilde{K}) \to H_1(M) \to 0.$$
Take a meridional disc $D_1 \subset M$ of $K_1$ and let $D_2$ denote $\varphi(D_1)$, where $\varphi: M \to M$ is the nontrivial covering transformation of the branched covering $\bar{p}: M \to S^3$; the covering transformation group is $\mathbb{Z}_2 = \{{\rm id}_M, \varphi\}$. We denote $D_1 \cup D_2$ by $\tilde{D}$. Because the kernel of the surjective homomorphism $H_1(N - \tilde{K}) \to H_1(N)$ is the image of the injective map $H_1(\partial\tilde{D}) \to H_1(N - \tilde{K})$, the short exact sequence above shows that
\begin{equation}\label{es1}
0 \to H_1(\partial\tilde{D}) \to H_1(M - \tilde{K}) \to H_1(M) \to 0
\end{equation}
is also exact. We should remark that the involution $\varphi$ induces automorphisms $\varphi_*$ on the homology groups in (\ref{es1}). Since the homomorphisms in (\ref{es1}) are induced by the inclusions, (\ref{es1}) is compatible with $\varphi_*$; i.e., the maps are $\mathbb{Z}_2$-equivariant.
\par Let $x \in H_1(\partial D_1)$ be a generator and set $y =  \varphi_*(x) \in H_1(\partial D_2)$. We use the same symbols $x, y$ for their images in $H_1(\partial{D})$ or $H_1(M - \tilde{K})$. 
We take the quotient of (\ref{es1}) by the $\varphi_*$-invariant part of $H_1(\partial\tilde{D})$ 
to obtain an exact sequence 
$$0 \to H_1(\partial\tilde{D})/(x+y) \to H_1(M - \tilde{K})/(x+y) \to H_1(M) \to 0.$$
Since $H_1(\partial\tilde{D})/(x+y) \cong \mathbb{Z}$ and $|H_1(M)| < \infty$, the rank of $H_1(M - \tilde{K})/(x+y)$ equals $1$, i.e., $(H_1(M - \tilde{K})/(x+y))/{\rm Tor}(H_1(M -\tilde{K})/(x+y)) \cong \mathbb{Z}.$ 
Hence there exists a surjective homomorphism $f: H_1(M - \tilde{K})/(x+y) \to \mathbb{Z}$, which satisfies $f \circ \varphi_* = -f$. 
Let $\bar{f}: \pi_1(X - \tilde{K}) \to \mathbb{Z}$ denote the composition
$$\pi_1(X - \tilde{K}) \to \pi_1(M - \tilde{K}) \to H_1(M - \tilde{K}) \to H_1(M - \tilde{K})/(x+y) \to \mathbb{Z}.$$
\par Let $m \in G_L$ be a meridian of a component of $L'$. Identifying $\langle b \rangle \subset D_\infty$ with $\mathbb{Z}$, we define $\tilde{f}: G_L \to D_\infty$ by
$$\tilde{f}(g) = \left\{ \begin{array}{ll} \bar{f}(g) & (g \in \pi_1(X - \tilde{K})),\\ a \bar{f}(m^{-1} g) & (g \not\in \pi_1(X - \tilde{K})). \end{array} \right.$$
Since $a^2 = e$, $\tilde{f}$ is well-defined as a map. 
Furthermore, we have $\bar{f}(mgm^{-1}) = f \circ \varphi_* (g) = f(g)^{-1} = \bar{f}(g)^{-1} \in D_\infty$ for $g \in \pi_1(X - \tilde{K})$. 
By this equality, 
we can easily check that $\tilde{f}$ is a group homomorphism. 
Because $\bar{f}$ is surjective and $\tilde{f}(m) = a$, the homomorphism $\tilde{f}: G_L \to D_\infty$ is surjective.
\end{proof}

The following is an immediate corollary of the proposition above, since any knot has an odd determinant. 

\begin{corollary}\label{cor3}
Let $L = \ell_1 \cup \ell_2$ be a $2$-component link. 
If $lk( \ell_1, \ell_2)$ is even, 
then $L$ admits a two-tone $D_\infty$-coloring that induces a surjective homomorphism from $G_L$ to $D_\infty$.
\qed
\end{corollary}

\section{Proof of theorems}\label{sec4}

In this section, we give proofs of the theorems stated in Introduction. 
To prove the theorems, we prepare the following two lemma. 

\begin{lemma}\label{lem1}
If a $2$-component link $L = \ell_1 \cup \ell_2$ is two-tone $D_n$-colorable for some odd  $n \ge 3$, then   
$lk( \ell_1, \ell_2)$ is even. 
\end{lemma}

\begin{proof}
Take a two-tone $D_n$-coloring $\gamma$ on a diagram of $L$ for some $n \ge 3$. 
Since $\gamma$ is two-tone, one component of $L$ is colored by $a_i$'s, and the other by $b_j$'s. 
Let $\ell_b$ be the component of $L$ such that each arc in a diagram of $\ell_b$ is colored by $b_j$'s by $\Gamma$. 
This $\ell_b$ is well-defined for $\Gamma$ independent of the choice of a diagram. See Remark~\ref{rmk21}. 

We can easily see that $L$ admits a diagram as depicted in Figure~\ref{Fig31}, 
where $D_b$ is a sub-diagram corresponding to $\ell_b$, $D_a$ is the remaining sub-diagram, and each box between $D_a$ and $D_b$ contains a vertical full twist (Figure~\ref{Fig31} (right)). 
For this $D_a \cup D_b$, we consider the arcs $\beta$ and $\beta'$ which are connected in $D_b$ as in Figure~\ref{Fig31} (left). 

\begin{figure}[htbt]
\centering
  {\unitlength=1mm
  \begin{picture}(90,48)
   \put(0,0){
  \includegraphics[width=.75\textwidth]{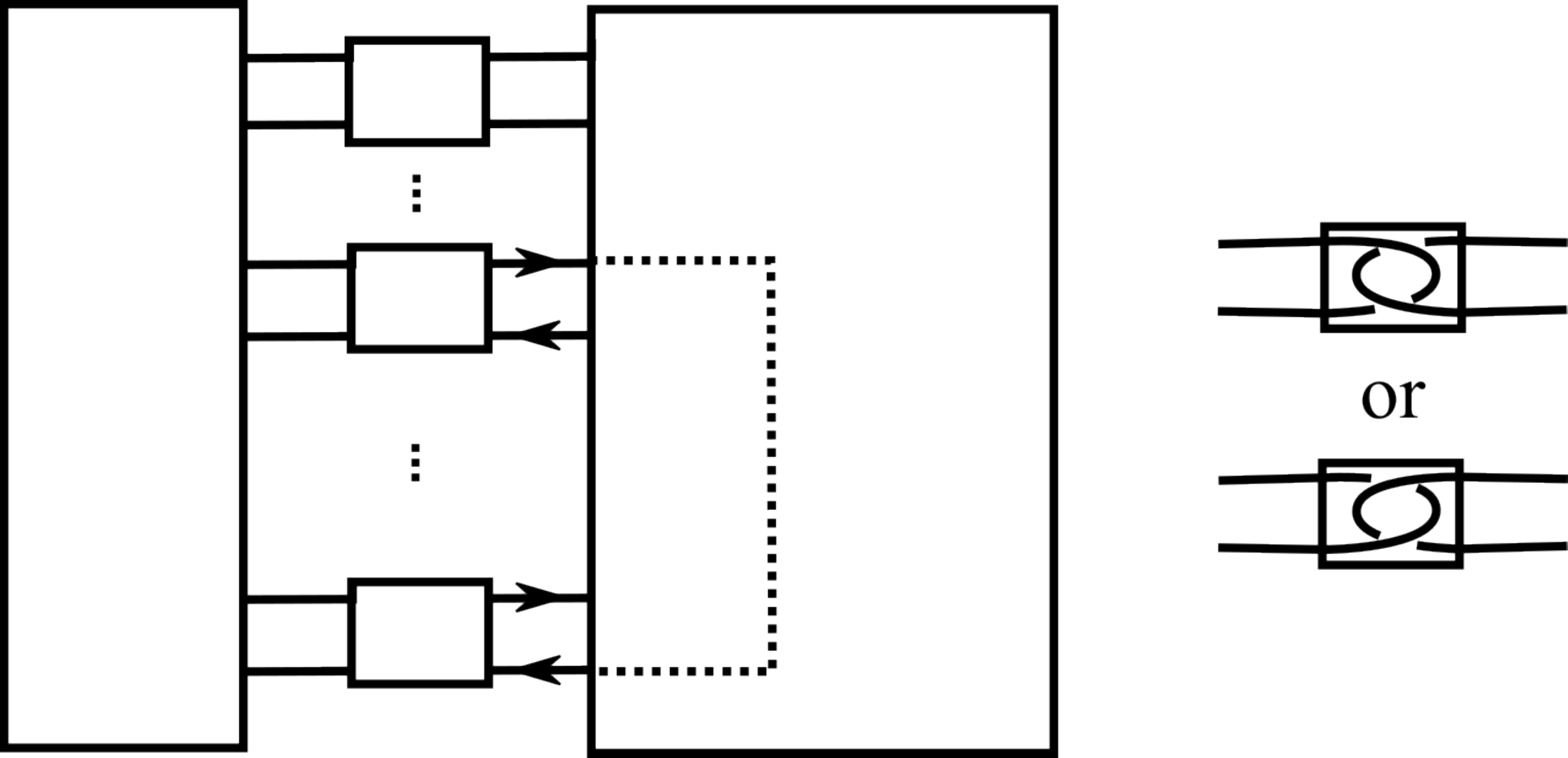}}
   \put(6,34){\large $D_a$}
   \put(45,34){\large $D_b$}
   \put(32,31){$\beta$}
   \put(32,1){$\beta'$}
  \end{picture}}
\caption{The diagram of $L$}\label{Fig31}
\end{figure}
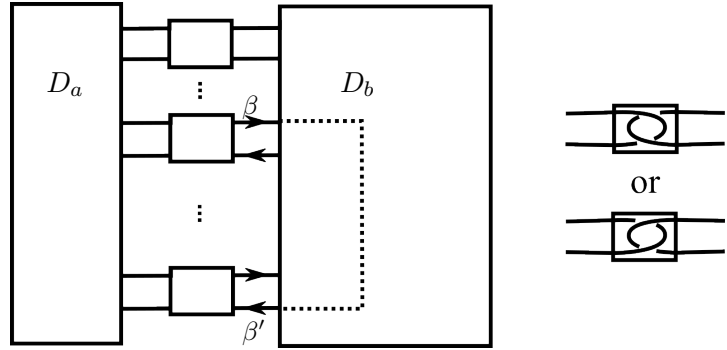

Since $\beta$ and $\beta'$ are connected in $D_b$, we see $\Gamma(\beta) = \Gamma(\beta')$ by Lemma~\ref{lem21}(1). 
On the other hand, letting $N$ be the number of the boxes (full twists) which $\ell_b$ runs through, if $\Gamma(\beta) = \Gamma(\beta')$, then $N$ has to be even. 
This is shown by applying Lemma~\ref{lem22} repeatedly for each box (full twist) together with $n$ is odd. 
The number $N$ is congruent to 
$lk (\ell_a , \ell_b)$ modulo 2, and so the lemma holds. 
\end{proof}

\begin{remark}\label{Remark5}
The lemma above can be extended as follows. 
If $L$ is two-tone $D_n$-colorable for some odd  $n \ge 3$, then the sublink $L_b$ of $L$ consisting of those components which are colored by $b_j$'s satisfies that, for every component $\ell \subset L_b$, $lk (\ell , L - L_b)$ is even.
\end{remark}

\begin{lemma}\label{lem3}
Let $L = \ell_1 \cup \ell_2$ be a $2$-component link. 
If $\det (L) =0$, then $lk( \ell_1, \ell_2)$ is even. 
\end{lemma}

\begin{proof}
Let $D$ be a diagram of $L$. 
Since $\det(L) = 0$, there exists a Fox $4$-coloring $\Gamma$ on $D$ which induces a surjective group homomorphism to $D_4$. 
By definition of Fox colorings, if $\Gamma(x)$ equals $a_0$ or $a_2$ (resp. $a_1$ or $a_3$) for an arc $x$ belonging to $\ell_i \; (i = 1,2)$, it holds for any arc $x$ of $\ell_i$. 
Then, we may assume
$$\Gamma(\{\text{arcs of $\ell_1$}\}) \subset \{a_1, a_3\} \quad \text{and} \quad \Gamma(\{\text{arcs of $\ell_2$}\}) \subset \{a_0, a_2\}.$$
For a crossing point of $D$, let $x$ be the over arc and $y, z$ the under arcs. 
Again, by definition of Fox colorings, we find that $\Gamma(y) = \Gamma(z)$ holds if and only if $x$ and $y$ belong to the same component. 
In particular, the colors of the under arcs at the crossing are changed if $x$ belongs to $\ell_1$ and $y$ to $\ell_2$. 
This implies that $D$ has an even number of such crossings, and hence the linking number $lk(\ell_1, \ell_2)$ is even.
\end{proof}

\begin{proof}[Proof of Theorem~\ref{Thm1}]
Let $L = \ell_1 \cup \ell_2$ a 2-component link. 
We show that all the following are equivalent. 
\begin{itemize}
\item[(i)] $lk( \ell_1, \ell_2)$ is even. 
\item[(ii)] $L$ is two-tone $D_n$-colorable for some odd  $n \ge 3$. 
\item[(iii)] $L$ is two-tone $D_\infty$-colorable.
\item[(iv)] The link group $G_L$ admits a surjective homomorphism to $D_n$ for every $n \ge 3$.
\item[(v)] The link group $G_L$ admits a surjective homomorphism to $D_\infty$.
\end{itemize}

We see that (i)$\Rightarrow$(iii) follows from Corollary~\ref{cor3} and (ii)$\Rightarrow$(i) follows from Lemma~\ref{lem1}. 

\medskip

\noindent
\underline{(iii)$\Rightarrow$(ii):} 
Suppose that $L$ is two-tone $D_\infty$-colorable, that is, a diagram of $L$ admits a two-tone $D_\infty$-coloring. 
Since there is a surjection from $D_\infty$ to $D_n$ for every $n \ge 3$ defined by $a \in D_\infty \mapsto a \in D_n$ and $b \in D_\infty \mapsto b \in D_n$, this implies that the diagram of $L$ admits a $D_n$-coloring for every $n$. 
By taking odd $n$ sufficiently large, the $D_n$-coloring uses at least two colors from $a_i$'s. 
Furthermore, by retaking $n$ to satisfy $ (n, \det(L))=1, (n, \det(\ell_1))=1$, and $(n,\det(\ell_2))=1$ if necessary, the coloring cannot come from Fox $n$-colorings on $L$, $\ell_1$, or $\ell_2$. Thus the coloring has to be two-tone, and so, $L$ is two-tone $D_n$-colorable for some odd  $n \ge 3$. 

\medskip

We also see that (i)$\Rightarrow$(v) follows from Corollary~\ref{cor3}. 

\medskip

\noindent
\underline{(v)$\Rightarrow$(iv):} 
By the surjection from $D_\infty$ to $D_n$ for every $n \ge 3$ defined as above, if the link group $G_L$ admits a surjective homomorphism to $D_\infty$, then the link group $G_L$ admits a surjective homomorphism to $D_n$ for every $n \ge 3$. 

\medskip

\noindent
\underline{(iv)$\Rightarrow$(i) or (ii):} 
Suppose that the link group $G_L$ admits a surjective homomorphism to $D_n$ for every $n \ge 3$.
Such a surjective homomorphism induces a $D_n$-coloring on a diagram of $L$ for $n \ge 3$ by considering the Wirtinger generators for the diagram. 
If $\det (L) =0$, then $lk( \ell_1, \ell_2)$ is even by Lemma~\ref{lem3}, and so (i) holds. 
If $\det (L) \ne 0$, then for some odd $n$ which is coprime to $\det(L), \det(\ell_1), \det(\ell_2)$, the $D_n$-coloring does not come from a Fox $n$-coloring, and so, it has to be two-tone. 
This implies (ii). 
\end{proof}

\begin{proof}[Proof of Theorem~\ref{Thm2}]
Let $L = \ell_1 \cup \ell_2$ be a 2-component link with $lk( \ell_1,\ell_2)$ is odd.

\noindent
(i) Then $L$ admits no two-tone $D_n$-colorings for any $n \ge 3$ by Theorem~\ref{Thm1} (by the contraposition of (ii)$\Rightarrow$(i)). 

\noindent
(ii)
By (i), if the link group $G_L$ admits a surjective homomorphism to $D_n$ for $n \ge 3$, then it is not induced from two-tone $D_n$-colorings. 
That is, the homomorphism must send Wirtinger generators to either the trivial element and reflections in $D_n$ or the trivial element and rotations in $D_n$. 
However, the latter cannot be surjective, and so, it is impossible. 
Therefore the homomorphism sends Wirtinger generators to either the trivial element and reflections in $D_n$. 
Such a homomorphism is induced from a Fox $n$-coloring on $\ell_1$, $\ell_2$ or $L$. 
\end{proof}

\begin{proof}[Proof of Theorem~\ref{Thm3}]
Let $L$ be a link with at least 3 components. 
We show that $G_L$ admits a surjective homomorphism to $D_n$. 

Consider sub-links of 2 components in $L$. 
If some of them, say $L' = \ell'_1 \cup \ell'_2$, satisfies that $lk (\ell'_1, \ell'_2)$ is even, then by Theorem~\ref{Thm2}, $G_{L'}$ admits a surjective homomorphism to $D_n$ and $L'$ is two-tone $D_n$-colorable for $n$. 
It follows that $G_L$ admits a surjective homomorphism to $D_n$ via a surjection $G_L \to G_{L'}$ and $L$ is two-tone $D_n$-colorable. 

Suppose that for all the 2 component sub-links of $L$, the linking numbers of the two components are odd. 
Then, by Lemma~\ref{lem3}, no such links have the determinant 0. 
Since $L$ has at least 3 components, we can consider a sub-link of $L$ with 3 components, say $L' = \ell_1 \cup \ell_2 \cup \ell_3$. 
For this link, $lk (\ell_1, \ell_2 \cup \ell_3)$ is even and $\det ( \ell_2 \cup \ell_3) \ne 0$ holds. 
Then, by Proposition~\ref{prop}, $G_{L'}$ admits a surjective homomorphism to $D_\infty$ and so a surjective homomorphism to $D_n$ for every $n \ge 3$. 
This implies that $G_L$ admits a surjective homomorphism to $D_n$ for every $n \ge 3$. 
\end{proof}

\begin{proof}[Proof of Corollary~\ref{CorG}]
Suppose that $L$ is two-tone $D_m$-colorable for some odd  $m \ge 3$. 
If $L$ is a link with 2 components, then $G_L$ admits a homomorphism to $D_n$ for every $n \ge 3$ by Theorem~\ref{Thm2} ((ii)$\Rightarrow$(iv)).  
If $L$ has at least 3 components, then $G_L$ admits a homomorphism to $D_n$ for every $n \ge 3$ by Theorem~\ref{Thm3}.  

Suppose that $G_L$ admits a surjective homomorphism to $D_n$ for $n \ge 3$. 
Then there is a $D_n$-coloring on a diagram of $L$. 
See Remark~\ref{rmk2}. 
If the coloring uses two-tone colors, 
then $L$ 
contains a two-tone $D_n$-colorable sub-link. 
Otherwise, since the homomorphism is surjective, the coloring comes from a nontrivial Fox $n$-coloring on a diagram of a sub-link of $L$ as in the proof of Theorem~\ref{Thm2}. 
 \end{proof}

\begin{remark}
For the proof of Lemma~\ref{lem3}, it is pointed out by the anonymous referee that the lemma is a direct consequence of the following two well-known formulas for Alexander polynomial $\Delta_L$:
\begin{itemize}
\item
$\Delta_L (1,1) = \pm lk (\ell_1, \ell_2)$ for a link $L = \ell_1 \cup \ell_2$ (\cite{Torres1953})
\item
$\det(L) = 2 | \Delta_L (-1, -1)|$ (\cite[Theorem 1]{HosokawaKinoshita1960}). 
\end{itemize}
(The second formula is a generalization of the Fox formula and a special case of the Mayberry-Murasugi formula \cite{MayberryMurasugi1982}, whose simple proof is given by Porti \cite{Porti2004}.) 
Moreover, the two formulas imply the stronger conclusion that $lk( \ell_1, \ell_2) \equiv 0 \pmod 2$ if and only if $\det(L) \equiv 0 \pmod 4$.
\end{remark}

\section{Finding two-tone colorings}

\begin{proof}[Proof of Proposition~\ref{prop1}]
Suppose that there exists a trivial component $\ell_0$ of a link $L$ and, for every component $\ell \subset L - \ell_0$, $lk (\ell_0 , \ell)$ is even. 
If a diagram of a link $L$ admits a two-tone $D_n$-coloring  for every odd $n \geq 3$ which assigns the arcs on $\ell_0$ to $a_i$'s and the other arcs to $b_j$'s, then so does any diagram of $L$. 
Thus, to prove the proposition, it suffices to show that a particular diagram of $L$ admits such a $D_n$-coloring. 

Now we take a diagram $D$ of $L$ depicted in Figure~\ref{Fig511}. 
In the figure, $D_0$ is a sub-diagram corresponding to $\ell_0$, which is a trivial knot diagram, and each box between $D_0$ and the remaining sub-diagram $D_b$ contains a vertical full twist (see Figure~\ref{Fig31} (right)).

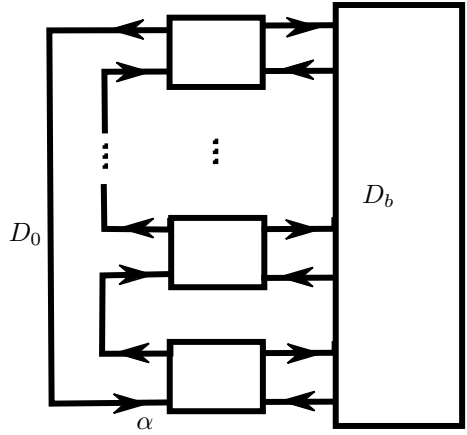
\begin{figure}[htbt]
\centering
  {\unitlength=1mm
  \begin{picture}(100,60)
   \put(20,0){\includegraphics[width=.44\textwidth]{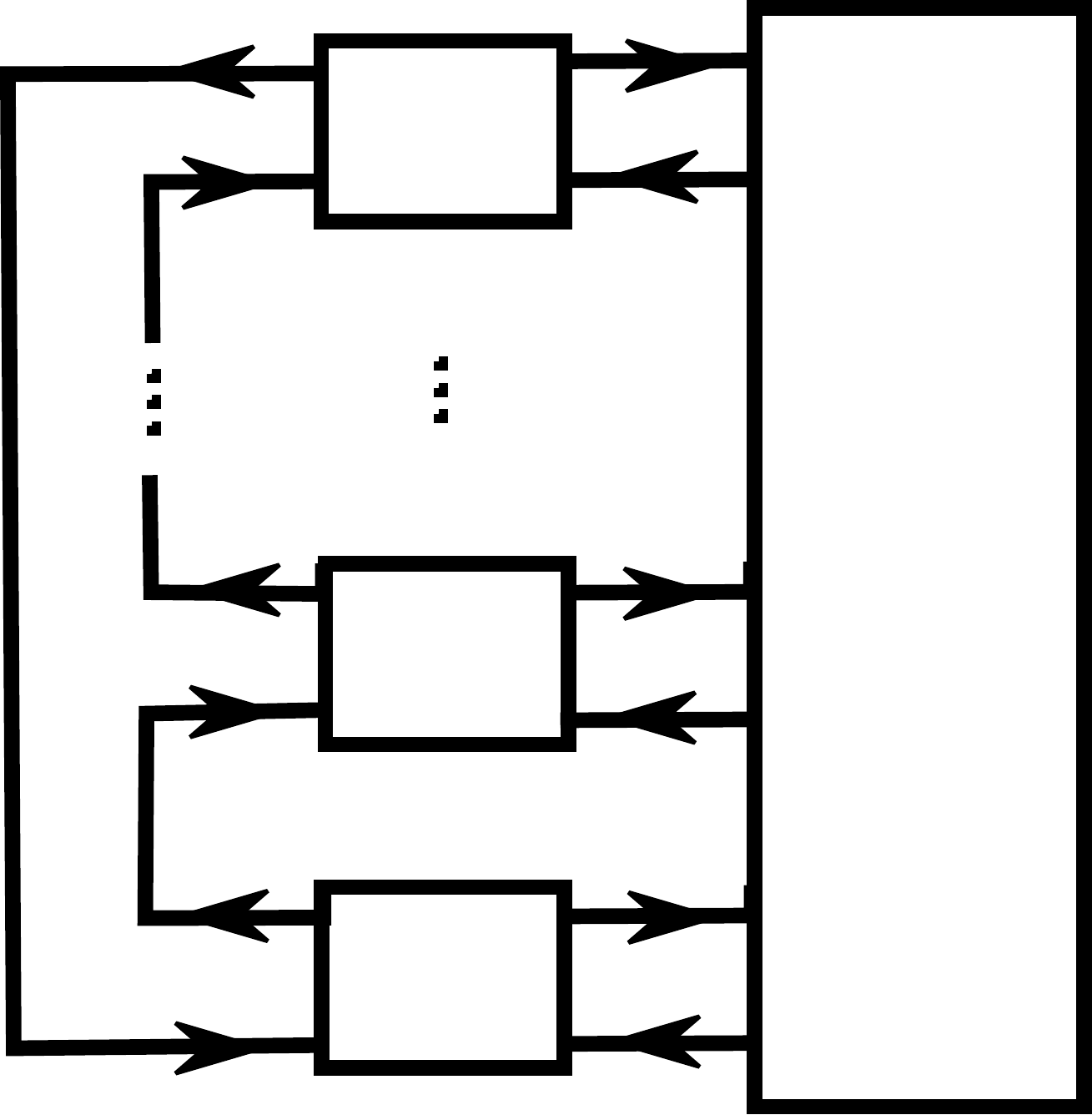}}
   \put(32,0){$\alpha$}
   \put(15,25){$D_0$}
   \put(62,30){$D_b$}
  \end{picture}}
\caption{The diagram D. Each box in the center contains a full twist}\label{Fig511}
\end{figure}

Consider the arc $\alpha$ in the figure, take an arc $\beta_i$ from each component of $L - \ell_0$, and assign $a_0$ to $\alpha$ and $b_1$ to $\beta_i$'s.
Let us show that this assignment induces a two-tone $D_n$-coloring. 

For the arc $\beta_i$, let $\ell$ be the component of $L-\ell_0$ containing $\beta_i$. 
Since $lk (\ell_0 , \ell)$ is even for every component $\ell \subset L - \ell_0$, due to Lemma~\ref{lem22}, the assigning $\beta_i$ to $b_1$ induces a $D_n$-coloring on $\ell$. 
In the same way, we can find a $D_n$-coloring on $L-\ell_0$. 

Note that, on the sub-diagram corresponding to each component of  $L-\ell_0$, 
an arc in the lower right of a box in the center is colored by $b_1$ or $b_{n-1}$. 
In particular, when the arc in the lower right is colored by $b_1$, then the arc in the upper right is colored in $b_{n-1}$, and vice versa. 

Thus, by Lemma~\ref{lem22}, for each component of $L - \ell_0$, the number of the boxes in the center with the arc in the lower right colored by $b_1$ is equal to the number of those with the arc colored in $b_{n-1}$. 

Let $m$ be the half of the linking number $lk(\ell_0, L - \ell_0)$. 
(Note that $lk(\ell_0, L - \ell_0)$ must be even, since $lk (\ell_0 , \ell)$ is even for each component $\ell$ of $L - \ell_0$.) 
Then the number of the boxes in the center with the arc in the lower right colored by $b_1$ is $m$ and the number of those with the arc colored in $b_{n-1}$ is also $m$.

Again by Lemma~\ref{lem22}, assigning $\alpha$ to $a_0$ induces assigning the arc in the upper left of the top box in center to $a_{0 - 2 ( m \cdot 1 + m \cdot (-1) )} = a_0$.
This implies that the assignment induces a $D_n$-coloring on the whole diagram. 
By construction, the $D_n$-coloring is obviously two-tone. 

Thus any diagram of $L$ admits a two-tone $D_n$-coloring  for every odd $n \geq 3$ which assigns the arcs on $\ell_0$ to $a_i$'s and the other arcs to $b_j$'s.
\end{proof}

\section*{Acknowledgement}
The authors would like to thank the anonymous referee for many useful comments and suggestions. 

\appendix

\section{}

The following was given by the anonymous referee for unifying and generalizing some of the arguments and results. 
The basic idea behind the approach is essentially identical with that of the proof of the key Proposition~\ref{prop}. 
However, it is quite different from the approach in the other parts. 

Recognized as before, 
a $D_n$-coloring ($n \in  \mathbb{Z}_{\ge 2} \cup \{ \infty \}$) of a link diagram $D$ representing a link $L$ is nothing other than a homomorphism, $\gamma$, from the link group $G_L := \pi_1 ( S^3 - L )$ to the dihedral group
\[ D_n = \langle a , b \mid a^2 , b^n , ( a b )^2 \rangle \cong \langle b \mid b_n \rangle \rtimes \langle a \mid a^2 \rangle, \]
that maps every meridian to a nontrivial element. 
Let $\nu : D_n \to \langle a \mid a^2 \rangle$ be the natural epimorphism. 
Then the coloring $\gamma$ corresponds to a Fox coloring or a two-tone coloring according to whether (i) $\nu \gamma$ maps every meridian to the generator $a$ or (ii) $\nu \gamma$ maps some meridian to $a$ and some meridian to the trivial element. 

Study of dihedral representations, more generally metableian representations, of link groups has a long history. In particular, a natural and useful viewpoint can be found in Hartley's article (\cite{Hartley1979}). 
The proof of the key Proposition~\ref{prop} fits this viewpoint. 
On the other hand, for Lemma~\ref{lem1} and Proposition~\ref{prop1}, which are intimately related with Proposition~\ref{prop}, the author give diagrammatic proofs, which have almost no relation with the proof of Proposition~\ref{prop}. 
Here, the following present unified proofs and generalizations of these results. 

\bigskip

\noindent
\textit{Homological proof of a generalization of Lemma~\ref{lem1} given in Remark~\ref{Remark5}}. 
By the assumption of the lemma, $G_L$ admits a two-tone $D_n$-representation $\gamma : G_L \to D_n$. 
Let $L_a$ and $L_b$ be the sublink of $L$ consisting of the components whose meridians are
mapped by $\nu \gamma$ to $a$ or 1, respectively. 
Since $\gamma$ maps the meridians of $L_a$ to order 2 elements, it descends to a homomorphism, which we
continue to denote by $\gamma$, from the quotient of $G_L$ by the normal closure of the squares of meridians of $L_a$. 
The latter group is the orbifold fundamental group of the orbifold, $\mathcal{O}$, with underlying space $S^3 - L_b$ with singular set $L_a$ of index 2. 
The double covering of $\mathcal{O}$ associated with the homomorphism $\nu \gamma : \pi_1^{orb} ( \mathcal{O}) \to \langle a \mid a^2 \rangle$ is the manifold $M - \tilde{L}_b$ 
where $M$ is the double branched covering of $S^3$ branched over $L_a$ and $\tilde{L}_b$ is the inverse image of $L_b$ in $M$. 
The fundamental group $\pi_1 ( M - \tilde{L}_b )$ is identified with the index 2 subgroup $\ker ( \nu \gamma )$ of $\pi_1^{orb} ( \mathcal{O})$, and the homomorphism $\gamma : \pi_1^{orb} ( \mathcal{O}) \to D_n$ restricts to an abelian representation $\tilde{\gamma} :\pi_1 ( M - \tilde{L}_ b ) \to \langle b \mid b^n \rangle < D_n$.

Now suppose to the contrary that there is a component $\ell$ of $L_b$ with $lk( \ell , L_a)$ odd. 
Then the inverse image $\tilde{\ell}$ of $\ell$ in $M - \tilde{L}_b$ is connected. 
Thus any two meridians of $\ell$, regarded as elements of $\pi_1 ( M - \tilde{L}_b)$, are conjugate in $\pi_1 ( M - \tilde{L}_b)$, and so their images by $\gamma$, which are equal to the images by the abelian homomorphism $\tilde{\gamma}$, are identical in $\langle b \mid b^n \rangle < D_n$. 
However, this is impossible, because for a meridian $\mu_\ell$ of $\ell$ and for a meridian $\mu_a$ of a component of $L_a$, we have $\gamma ( \mu_a \mu_\ell \mu_a^{-1} )=  \gamma ( \mu_\ell)^{-1}
\ne \gamma( \mu_\ell )$, though $\mu_a \mu_\ell \mu_a^{-1}$ is also a meridian of $\ell$. 
(Here the inequality follows from the assumption that $n \ge 3$ is odd.) \qed

\bigskip

Though the above proof is lengthy, it ties up with the proof of Proposition~\ref{prop} and it leads to a simple proof of the following generalization of Proposition~\ref{prop1}.

\begin{proposition} 
Let $L = L_0 \cup L_1$ be a link in $S^3$ satisfying the following conditions.
\begin{itemize}
\item[(1)] $\det ( L_0 )  = 1$.
\item[(2)] $L_1$ is non-empty, and every component of $L_1$ has an even linking number with $L_0$.
\end{itemize}
Then there is a two-tone epimorphism from $G_L$ to $D_\infty$ for which $L_a = L_0$ and $L_b = L_1$, 
where $L_a$ and $L_b$ are the sublinks of $L$ as in the ``homological proof''. 
\end{proposition}

\begin{proof}
Let $M$ be the double branched covering of $S^3$ branched along $L_0$ and $\tilde{L}_1$ the inverse image of $L_1$ in $M$. 
The assumptions imply that $H_1 ( M - \tilde{L}_1 )$ is a free abelian group with basis $\{ \mu_i , \mu'_i  \mid 1 \le i \le r \}$ such that the homomorphism $\tau$ induced by the covering translation switches $\mu_i$ with $\mu'_i$ for each $i$. 
(Here $r$ is the number of components of $L_2$, $\mu_i$ and $\mu'_i$ are meridians of the components of $\tilde{L}_1$ that are mapped to the $i$-th component of $L_2$.) 
Let $Q$ be the semi-direct product of $H_1 ( M - \tilde{L}_1)$ with the order 2 cyclic group  $\langle a \mid a^2 \rangle$, where the action of the latter group on the first group is given by $\tau$. 
Then $Q$ is a quotient of the link group $G_L$. 
(In fact it is a quotient of the orbifold fundamental group of the orbifold $\mathcal{O}$ with underlying space $S^3 - L_1$ with singular set $L_0$ of index 2, as defined in the homological proof of Lemma~\ref{lem1}.) 
The proposition now follows from the fact that there is an epimorphism from $Q$ to $D_\infty$ defined by $a \mapsto a$, $\mu_i \mapsto b$ and $\mu'_i = \tau( \mu_i ) = a \mu_i a^{-1} \mapsto b^{-1}$. 
\end{proof}

The above proof and that of Lemma~\ref{lem1} work for links in a $\mathbb{Z}$-homology 3-sphere. 
Moreover, the same argument also imply the following further generalization. 

\begin{proposition}
Let $L = L_0 \cup L_1$ be a link in a $\mathbb{Z}/2\mathbb{Z}$-homology 3-sphere S, and $n \ge 2$ an integer, satisfying the following conditions.
\begin{itemize}
\item[(1)] 
The double branched covering $M$ of $S$ branched over $L_0$ is a $\mathbb{Z}/n\mathbb{Z}$-homology 3-sphere.
\item[(2)] 
$L_1$ is non-empty, and every component of $L_1$ has the trivial mod 2 linking number with $L_0$.
\end{itemize}
Then there is a two-tone epimorphism from $G_L$ to $D_n$ for which $L_a = L_0$ and $L_b = L_1$, where $L_a$ and $L_b$ are the sublinks of $L$ as in the ``homological proof''. 
\end{proposition}

In fact, the assumption that $L$ is a link in a $\mathbb{Z}/2\mathbb{Z}$-homology 3-sphere implies that there is a unique double branched covering branched along $L$, and the two conditions imply that $H_1 ( M - \tilde{L}_1; \mathbb{Z}/n\mathbb{Z})$ is the free $\mathbb{Z}/n\mathbb{Z}$-module that has a base consisting of meridians $\{ \mu_i , \mu'_i  \mid 1 \le i \le r \}$, such that the homomorphism $\tau$ induced by the covering translation switches $\mu_i$ with $\mu'_i $ for each $i$.

\bigskip

There are possible future problems (also given by the anonymous referee): It would be nice if one could give a unified diagrammatic proof to all of the results in the paper, including the key Proposition~\ref{prop} and the results in the appendix. 
If successful, then it might bring our mathematical community a new deep insight into the link diagrams. 

Also the results in this paper might give a hint to the following natural question: 

\medskip

\noindent
\textbf{Question.} For $n = 1$ or $2$, the “greatest common quotient” of the $n$-component link groups is the free abelian group $\mathbb{Z}_n$. 
For $n \ge 3$, is there a non-abelian group $G$ bigger than $\mathbb{Z}_n$ (i.e., a non-commutative group with abelianization $\mathbb{Z}_n$), for which every $n$-component link group admits a (canonical) epimorphism onto $G$?

\medskip

If such a group $G$ exists, then by considering the $G$-coverings of link complements, one may be able to construct a link invariant stronger than the Alexander invariants, which are defined by using $\mathbb{Z}_n$-coverings.

\bibliographystyle{abbrv}
\bibliography{2024IIMS.bib}

\end{document}